\documentclass[12pt,reqno,oneside]{amsart}

\usepackage{amsmath}
\usepackage{amssymb}
\usepackage{amsthm}
\usepackage{graphics}
\usepackage{eucal}
\usepackage{bm}
\usepackage[colorlinks]{hyperref}
\usepackage[capitalize,nameinlink,noabbrev]{cleveref}
\usepackage{mathrsfs}
\usepackage[normalem]{ulem}
\usepackage{color}
\usepackage[dvipsnames]{xcolor}
\usepackage{mathtools}
\usepackage{geometry}
\newtheorem{theorem}{Theorem}
\newtheorem{definition}{Definition}[section]
\newtheorem{proposition}{Proposition}[section]
\newtheorem{lemma}{Lemma}[section]

\newtheorem{remark}{Remark}[section]
\newtheorem{problem}{Problem}[section]
\newtheorem{conjecture}{Conjecture}[section]
\newtheorem{dream theorem}{Dream theorem}[section]
\DeclareMathOperator{\inj}{inj}
\DeclareMathOperator{\Ad}{Ad}
\numberwithin{equation}{section}

\newcommand{\SL}{\operatorname{SL}}
\newcommand{\diag}{\operatorname{diag}}

\newcommand{\supp}{\operatorname{supp}}

\newcommand{\bx}{\mathbf{x}}

\newcommand{\bp}{\mathbf{p}}

\newcommand{\by}{\mathbf{y}}

\newcommand{\Z}{\mathbb{Z}}
\newcommand{\R}{\mathbb{R}}

\usepackage{mathtools}
\usepackage{tikz-cd}
\usepackage{graphicx}

\author{Liyang Shao}
\address[Liyang Shao]{Department of Mathematics, University of California-Berkeley, Berkeley, USA}
\email{liyang\_shao@berkeley.edu}
\title[Nonlinear Khintchine]{Khintchine dichotomy for images of self-similar measures and almost random walks}
\date{}

\usepackage{environ}

\begin{document}
\sloppy
%\large

\begin{abstract}
    Following the celebrated breakthrough in B{\'e}nard-He-Zhang  \cite{Benard_He_Zhang_2026}\cite{BHZ_2026}, we prove Khintchine dichotomy for push-forwards of self-similar measures on $\R^d$ under uniform $C^{1+\alpha}$ maps. For this, we also study almost random walks on $\SL(d+1,\R)/\SL(d+1,\Z)$ and prove their effective equidistribution. 
\end{abstract}

\maketitle
%\centering
\section{Introduction}

Denote by $\text{Sim}(\R^d)$ the group of similarities of $\R^d,$ which are of the form $$\phi:\bx\mapsto r_{\phi}O_\phi\bx+\mathbf b_\phi,r_\phi>0,\mathbf b_\phi\in\R^d,O_\phi\in O(d).$$

Now we consider an irreducible probability measure $\lambda$ on $\text{Sim}(\R^d)$,  i.e. no proper affine subspace of $\R^d$ is invariant under $\lambda$-almost
every $\phi.$ Further, we assume that $\lambda$ satisfies the finite moment condition, i.e. there exists $\epsilon>0$ so that \begin{equation}\int\bigl(r_\phi^\epsilon+r_\phi^{-\epsilon}+\Vert\mathbf b_\phi\Vert^\epsilon\bigr)\,d\lambda(\phi)<\infty,\label{lambda-finite}\end{equation} and is also contractive on average, i.e. \begin{equation}\label{contra}
    l:=-\int\ln r_\phi d\lambda(\phi)>0.
\end{equation}

Then by \cite[Theorem 2.5]{Bougerol_Picard_1992} there exists a unique probability measure $\sigma$ on $\R^d$ so that $\sigma$ is $\lambda$-stationary \begin{equation}
    \label{linear-sta}\int\phi_*\sigma d\lambda(\phi)=\sigma.
\end{equation}

A canonical example of such self-similar measure would be the Hausdorff measure $\mathcal H^{\dim K}$ restricted to middle-thirds Cantor set $K$, where the corresponding $\lambda$ is evenly distributed on the two affine transformations $\phi_1:x\to x/3$ and $\phi_2:x\to x/3+2/3.$

The Diophantine properties of points distributed according to self-similar measures have been extensively studied over the last a few decades. Let's first recall the main objects of such study, \textit{$\Psi$-approximable} vectors $$W(\Psi):=\left\{\bx\in \R^d:\Vert q\bx-\bp\Vert_{\infty}<\Psi(q)\text{ for infinitely many }(\bp,q)\in\Z^{d}\times\Z_{>0}\right\}.$$ 

Classical Khintchine's dichotomy on Lebesgue measure says that, $W(\Psi)$ is either null or full Lebesgue measure, and each scenario can be simply determined by whether $\sum_{q>0}\Psi(q)^d<\infty$ or not \cite{Khintchine1926}. 

For Khintchine dichotomy on fractal sets, a motivating question is Mahler’s 1984 question about rational approximation to points of the middle-thirds Cantor set \cite{Mahler_1984}. Following this question, extremality was established for self-similar measures with the open set condition in \cite{Weiss2001} and subsequently developed within the framework of (absolutely) friendly measures in \cite{KLW,Pollington_Velani_2005}. %via studying `friendly' measure.  

More recently, elaborating Benoist-Quint  \cite{BenoistQuint2012,BenoistQuint2013II,BenoistQuint2013III} and Eskin-Mirkazhani-Mohammadi \cite{EMM2015}, a theory of expanding measures by Simmons-Weiss \cite{Simmons_Weiss_2019} and Prohaska-Sert-Shi \cite{Prohaska_Sert_Shi_2023} established a systematic viewpoint of random walks on simple Lie groups %$\SL(d+1,\R)/\SL(d+1,\Z)$ 
towards Diophantine approximation on fractals. In a related development, Khalil-Luethi \cite{Khalil_Luethi_2023} connects Khintchine's dichotomy with equidistribution of trajectory $a(t)u(\mathbf s)\Z^{d+1}$ and proves Khintchine's dichotomy for a large class of self-similar measures. 

It is worth noting that Fourier-analytic approaches yield convergence results in Yu \cite{Yu_2021} and convergence and divergence results in Datta-Jana \cite{Datta_Jana_2024}, under suitable Fourier-dimensional hypotheses. %Such approaches are related to equidistribution on $\SL(d+1,\R)/\SL(d+1,\Z)$ \cite{Datta_Jana_2024}. 

The celebrated breakthrough in B{\'e}nard-He-Zhang \cite{Benard_He_Zhang_2026}(in the case of $\R$, and later in \cite{BHZ_2026} for the case of $\R^d$.) established Khintchine dichotomy for all self-similar measures $\sigma$ driven by irreducible probability measures $\lambda$ on Sim($\R^d$). They proved this dichotomy by establishing effective equidistribution and double equidistribution of expanding translates of \(\sigma\) by the method of dimension bootstrap%, where the key ingredient is a multislicing projection theorem as described in \cite[Section 7]{BHZ_2026}.

It is somehow natural to expect Khintchine dichotomy on images of self-similar measures. One motivation is the polynomial Fourier decay of nonlinear images of self-similar measures, which can be dated back to Kaufman \cite{Kaufman1984} has recently been established for analytic non-affine images in dimension one \cite{AlgomChangWuWu2025, BakerBanaji2025}, and in higher dimensions \cite{BanajiYu2026Qualitative}. Indeed, for example, with sufficiently large Fourier decay exponent, one can derive the  convergence part of Khintchine dichotomy(see, e.g. \cite{Yu_2021}).

The main goal of this paper is to establish Khintchine's dichotomy for the push-forward of self-similar measures under $C^{1+\alpha}$-maps from $\R^d$ to $\R^d$, with mild conditions on non-degeneracy and boundedness. %Note that this is not an immediate consequence of the self-similar case: a smooth change of variables generally destroys the similarity structure and does not preserve rational approximation.

The specific setting is as follows. Let $h:\R^d\to\R^d$ be a $C^{1}$-map, defined on a neighborhood of the support of $\sigma$ with \begin{equation}|\det Dh|\geq c>0,\quad\Vert Dh\Vert_\infty\leq C<\infty.\label{bdd-assumption0}\end{equation}
Further, for  $\alpha\in(0,1],$ we require the following uniform Hölder bounds on $\supp\sigma$: for some $H>0$ and all $\mathbf v,\mathbf w\in \supp\sigma$,
\begin{equation}\label{uniform-holder}
\begin{split}
\Vert Dh(\mathbf v)-Dh(\mathbf w)\Vert&\leq H\Vert\mathbf v-\mathbf w\Vert^\alpha,\\
\Vert h(\mathbf v)-h(\mathbf w)-Dh(\mathbf w)(\mathbf v-\mathbf w)\Vert&\leq H\Vert\mathbf v-\mathbf w\Vert^{1+\alpha}.
\end{split}
\end{equation}
Note that these bounds are automatic when $h$ is $C^{1+\alpha}$ and $\supp\sigma$ is compact.%Consider $$\nu=h_*\sigma.$$

We establish the following Khintchine's dichotomy theorem for $\nu=h_*\sigma,$

\begin{theorem}\label{main: K}
    Let $\sigma$ be a self-similar measure on $\R^d$ driven by an irreducible probability measure $\lambda$ on Sim$(\R^d)$ that satisfies \eqref{lambda-finite} and \eqref{contra}, and $h:\R^d\to\R^d$ be a $C^{1}$-map that satisfies \eqref{bdd-assumption0} and \eqref{uniform-holder}. Denote $\nu=h_*\sigma.$ Then for any non-increasing function $\Psi:\Z_{>0}\to[0,\infty),$ we have$$\nu(W(\Psi))=\begin{cases}0,&\sum_{q=1}^\infty\Psi(q)^d<\infty,\\
1,&\sum_{q=1}^\infty\Psi(q)^d=\infty.
\end{cases}$$
\end{theorem}

\begin{remark}
    The main theorem is not an immediate consequence of \cite{BHZ_2026}; push-forwards under smooth maps usually destroy the self-similarity of a measure, thus giving difficulty in following exactly the route in \cite{BHZ_2026}. The key ideas will be introduced in \cref{idea}.
\end{remark}

\begin{remark}
    We may assume, without loss of generality, that $h$ is orientation-preserving. Indeed, let $J=\diag\{-1,1,\cdots,1\}\in O(d),$ then $J\bx\in W(\Psi)\iff\bx\in W(\Psi).$ Then we get a new orientation-preserving map $\tilde h$ by, if necessary, flipping the orientation sign of $h$ via $J$ on each component of $\supp\sigma,$ and by previous arguments $\tilde\nu(W(\Psi))=\nu(W(\Psi))$ where $\tilde\nu=\tilde h_*\sigma.$ Further, this new $\tilde h$ still satisfies \eqref{bdd-assumption0} and \eqref{uniform-holder}, since \eqref{bdd-assumption0}\eqref{uniform-holder} for $h$ tells us that the components of $\supp\sigma$ where $h$ has different orientation signs are at least $(\frac{c}{C^{d-1}H})^{1/\alpha}$ apart. Indeed, suppose $h$ has opposite orientation signs at $\bx,\by\in\supp\sigma,$ then $$\Vert \bx-\by\Vert^\alpha\geq H^{-1}\Vert Dh(\bx)-Dh(\by)\Vert\geq \frac{c}{C^{d-1}H},$$where the last inequality is true by bounding singular values of $Dh$ at $\bx$ and $\by.$

    From now on, we assume that $h$ is orientation-preserving:\begin{equation}\det Dh\geq c>0,\quad\Vert Dh\Vert_\infty\leq C<\infty.\label{bdd-assumption}\end{equation}
\end{remark}

\subsection{Structure of the paper}

In \cref{def-sec}, we recall the necessary notions, and how Khintchine's dichotomy can be reduced to (double) equidistributions, as developed in \cite{Khalil_Luethi_2023,BHZ_2026}.

In \cref{arw-sec}, we introduce the notion of almost random walks. Then we show their recurrence and equidistribution via Margulis height function and results from \cite{BHZ_2026}.

In \cref{mainproof-sec}, we first give an example of almost random walks associated to images of self-similar measures. Then we piece \cref{arw-sec} and the proof idea of \cite[Proposition 6.1]{Benard_He_Zhang_2026} altogether to prove the (double) equidistribution theorems.

Finally in \cref{future-sec} we give some remarks about further questions to explore.

\subsection*{Acknowledgement}The author thanks his advisors Svetlana Jitomirskaya and Amir Mohammadi for suggesting this problem and providing continuous support and guidance. The author thanks Anton Gorodetski, Zuo Lin and Wenyu Pan for many helpful discussions. The author also thanks Wenyu Pan for explaining her joint work with Li, Tong and Xu \cite{LiPanTongXu2026}.

This research was supported by NSF DMS-2052899, DMS2155211, and Simons 896624.

\subsection*{AI acknowledgement} A generative LLM was used for proofreading and literature searching. All main mathematical contributions are from the human authors.
%Combining these two approximation steps, one could get Khintchine dichotomy via estabilishing equidistribution according to \cite{BHZ_2026}.
%In this paper, we aim to study a type of process on $X=G/\Gamma,$ where $\Gamma$ is a lattice in a Lie group $G,$ which is called \textit{$\mu$-almost random walks} and will be defined in \cref{def-sec}. In particular, we will give an effective recurrence and equidistribution result of such process. As a consequence of 

\section{Setup and proof strategy}
\label{def-sec}

We follow the route in \cite{BHZ_2026} to prove Khintchine's dichotomy via (double) equidistribution of $\nu$.

First of all, we may assume $\lambda$ is supported on orientation-preserving similarities by \cite[Lemma 5.3]{Benard_He_Zhang_2026}.

We look at $\text{SL}(d+1,\R),$ 
$d\geq1$. Let $\Lambda$ be a lattice in $\text{SL}(d+1,\R)$ and $X:=\text{SL}(d+1,\R)/\Lambda.$ We write $m_X$ for the Haar probability measure. Derivatives and the corresponding norms are taken with respect to a fixed basis of the Lie algebra of $\text{SL}(d+1,\R)$.

 $$a(t)=\diag\{t^{\frac{1}{d+1}}I_d,t^{-\frac{d}{d+1}}\},u(\mathbf s)=\begin{pmatrix}
    I_d&\mathbf s\\\mathbf 0&1
\end{pmatrix},\mathbf s\in\R^d,\ t>0,$$ both of which are in $\SL(d+1,\R).$

Also
$$p:=\left\lceil\frac{1}{2} \dim \text{SO}(d+1)\right\rceil.$$

We focus on the function space $B_{\infty,p}(X),$ which denotes the smooth functions on $X$ whose $\infty$-norms of derivatives of order up to $p$ are bounded. The corresponding norm in $B_{\infty,p}(X)$ is denoted by $\mathcal S_{\infty,p}(\cdot)$.

In \cite{BHZ_2026} the following equidistribution was proved.\begin{theorem}[Theorem 1.2 in \cite{BHZ_2026}]\label{Rd-equi}
    There exist constants $C,c>0$ that depend only on $\Lambda,\sigma$ so that for any $f\in B_{\infty,p}(X),x\in X,t>1,$

    $$\left|\int f(a(t)u(\mathbf s)x)d\sigma(\mathbf s)-m_X(f)\right|\leq C\mathcal S_{\infty,p}(f)\text{inj}(x)^{-1}t^{-c}.$$
\end{theorem}

Building upon the previous theorem, we prove the following main theorems

\begin{theorem}\label{main: single}
    There exist constants $C,c,\Delta>0$ that depend only on $h,\Lambda,\sigma$ so that for any $f\in B_{\infty,p}(X),x\in X,t>1,$

    $$\left|\int f(a(t)u(\mathbf s)x)d\nu(\mathbf s)-m_X(f)\right|\leq C\mathcal S_{\infty,p}(f)\text{inj}(x)^{-\Delta}t^{-c}.$$
\end{theorem}

Due to \cite[Theorem 9.1]{Khalil_Luethi_2023}, the above equidistribution theorem gives convergence part of Khintchine's
dichotomy.

Fix $x_0=e\Lambda\in X$. Recall the following definition in \cite[Section 8]{BHZ_2026}, which is inspired by \cite{KleinbockShiWeiss2017}.

\begin{definition}
    Given a probability measure $\xi$ on $\mathbb{R}^d$, bounded continuous functions $f_1, f_2 : X \to \R$, and times $t_2 \ge t_1 > 0$, we introduce
\begin{equation}
\Delta_{f_1,f_2}^\xi(t_1,t_2) := \left| \int_{\R^d} f_1(a(t_1)u(\mathbf s)x_0) f_2(a(t_2)u(\mathbf s)x_0) \, \mathrm{d}\xi(\mathbf s) - m_X(f_1)m_X(f_2) \right|.
\end{equation}
\end{definition}

We will further upgrade \cref{main: single} to the following double equidistribution. Such double equidistribution theorem would give Schmidt's estimate(a quantitative version of Khintchine's dichotomy; see \cite{Schmidt1960}) and thus lead to divergence part of Khintchine's dichotomy, as described in \cite[Section 9]{BHZ_2026}

\begin{theorem}\label{main: double}
    For every $\eta > 0$, there exist $C,c > 0$, depending also on $h,\Lambda,\sigma$, such that for all $t_1,t_2 > 1$ with $t_2 \ge t_1^{1+\eta}$ and $f_1, f_2 \in B_{\infty,p}(X)$, we have
\begin{equation}
\Delta_{f_1,f_2}^\nu(t_1,t_2) \le C\mathcal{S}_{\infty,p}(f_1)|m_X(f_2)|t_1^{-c} + C\mathcal{S}_{\infty,p}(f_1)\mathcal{S}_{\infty,p}(f_2)t_2^{-c}.
\end{equation}
\end{theorem}

\textit{Proof of \cref{main: K}.} Taking $\Lambda=\text{SL}(d+1,\Z)$, \cref{main: single} and \cref{main: double} give Khintchine's dichotomy by \cite[Theorems 9.2 and 9.3]{BHZ_2026}

\section{Almost random walks}\label{arw-sec}

Before proving equidistribution theorems, we study almost random walks, which need not be genuine random walks but have an approximate convolution structure. This will unveil an important difference between $\sigma$ and $\nu.$ We will show their effective recurrence and equidistribution. The recurrence result will help us prove \cref{main: single} and \cref{main: double}. The equidistribution result could be of independent interest.

Consider a Lie group $G,$ and $\mu$ is a probability measure on $G.$ $G$ acts on a metric space $Z$ and $Y$ is a space of functions on $Z$ with norm $\Vert\cdot\Vert.$  We consider almost random walks in the following sense.

\begin{definition}
\label{almost-def}     $\zeta_n=\lambda_n*\delta_x$ ($\lambda_0=\delta_e$ and $\zeta_0=\delta_x$), where $(\lambda_n)_{n\geq0}$ is a sequence of probability measures on $G,$ is called a $\mu$-almost random walk starting from $x\in Z$ if it satisfies
     \begin{itemize}
         \item Finite moment: there are some $K,\gamma>0$ so that \begin{equation}
             \label{almost-moment}\int_G\Vert\Ad(g)\Vert^\gamma d\lambda_n(g)\leq e^{Kn}<\infty,
         \end{equation}

         \item Exponentially close to convolution: there exist $\theta\in(0,1)$ and $C,c>0$ so that for every  $n\geq1$, with $m=\lceil\theta n\rceil$, and every $f\in Y,$ we have \begin{equation}
        |\zeta_n(f)-\mu^{*m}*\zeta_{n-m}(f)|\leq C\Vert f\Vert e^{-cn}.\label{almost-conv}
    \end{equation}
     \end{itemize}

\end{definition} 

From now on we fix $$G=\SL(d+1,\R),Z=X=G/\Lambda,Y=B_{\infty,p}(X).$$
For a matrix $A\in\text{GL}^{+}(d,\R)$, we denote $$k_{A}:=\diag\{|\det A|^{-\frac{1}{d+1}}A,|\det A|^{-\frac{1}{d+1}}\}\in\text{SL}(d+1,\R).$$

 Now let $\mu$ be the image of $\lambda$ under the anti-isomorphism $\phi\mapsto g_\phi:=k_{O_\phi}^{-1}a(r_\phi^{-1})u(\mathbf b_\phi)\in \text{SL}(d+1,\R)$.
 We will interchange between using $\mu$ and $\lambda$ now and then.

We use the convention $\psi\phi=\psi\circ\phi.$ Correspondingly, $g_{\phi\psi}=g_\psi g_\phi.$

\subsection{Effective recurrence}

This subsection will utilize Margulis height functions to show effective recurrence. Margulis function, generally aiming to show that trajectories of some dynamical
system spend most of the time in compact sets, stems from Eskin-Margulis-Mozes \cite{EMM1998} and Eskin-Margulis \cite{EM2004}, and has subsequently seen many important constructions including random walks on semisimple Lie groups by Benoist-Quint \cite{BenoistQuint2012} and random walks on moduli spaces by Eskin-Mirzakhani-Mohammadi \cite{EMM2015}.

The height function we are going to use is from B{\'e}nard-He-Zhang \cite{BHZ_2026}, which is essentially constructed in Prohaska-Sert-Shi \cite{Prohaska_Sert_Shi_2023}.

We first recall two lemmas about height functions. %as described in \cite[Section 5]{BHZ_2026}, which is essentially from the framework of \cite{Prohaska_Sert_Shi_2023}.

\begin{lemma}[proof of Proposition 5.1 in \cite{BHZ_2026}]\label{h-fcn:def}
    There exists a proper function $\beta:X\to[1,+\infty)$ that satisfies\begin{itemize}
        \item There exist $M,\kappa>0$ so that for every  $n\geq1$ and $x\in X$\begin{equation}
            \label{margulis}\mu^{*n}*\delta_x(\beta)\leq e^{-\kappa n}\beta(x)+M,
        \end{equation}
        \item For every $g\in G,x\in X$\begin{equation}
            \label{growth-rate}\beta(gx)\leq\Vert\Ad(g)\Vert^{O(1)}\beta(x).
        \end{equation}
    \end{itemize}
\end{lemma}

\begin{lemma}[Lemma 5.4 in \cite{BHZ_2026}]\label{h-inj}

For any proper function $\beta':X\to[1,+\infty)$ that satisfies the two conditions in \cref{h-fcn:def}, there exist $c,C>0$ so that for all $x\in X,$ \begin{equation}
    \beta'(x)^{-c}\ll\inj(x)\ll \beta'(x)^{-C}.
\end{equation} 

\end{lemma}

To utilize the close-to-convolution property \eqref{almost-conv} for smooth functions, we mollify the height function as follows,
\begin{lemma}\label{sm}
    We may assume that the height function $\beta$ in \cref{h-fcn:def}(that satisfies \eqref{margulis} and \eqref{growth-rate}) is smooth. Further, for every differential operator $D,$ one has\begin{equation}\label{diff-estimate}|D\beta|\ll_D\beta.\end{equation}
\end{lemma}

\begin{proof}
    Take a nonempty bounded open set $B\subset G,$ and a smooth bump function $\rho_B$ on $B.$ Now define$$\tilde \beta(x):=\int_G\rho_B(g)\beta(g^{-1}x)dg,$$where $dg$ is the (left) Haar measure on $G$. Then for some constant $C_B>0$ \begin{equation}\label{cb-def}C_B^{-1}\tilde\beta(x)\leq\beta(x)\leq C_B\tilde\beta(x),\forall x\in X.
    \end{equation}Note that $C_B\to1$ when $B$ shrinks. Moreover, for $u\in\text{Lie}(G),\tilde\beta(\exp(tu)x)=\int_G\rho_B(\exp(tu)g)\beta(g^{-1}x)dg$ for small $t$ by left-invariance of Haar measure. As a consequence $\tilde\beta$ is smooth and by iterating the process above, for every (left) differential operator $D,$ $$|D\tilde\beta(x)|\leq\int_G|D\rho_B(g)|\beta(gx)dg\ll_D\beta(x)\overset{\eqref{cb-def}}{\ll}\tilde\beta(x),$$which gives \eqref{diff-estimate}. Moreover, $\tilde\beta$ still satisfies \eqref{growth-rate} by definition, while \eqref{cb-def} gives $\mu^{*n}*\delta_x(\tilde\beta)\leq C_B^{2}(e^{-\kappa n}\tilde\beta(x)+M)$, thus \eqref{margulis} is still satisfied by shrinking $B$ if necessary.
\end{proof}

Then using such a height function, we prove the following effective recurrence of $\mu$-almost random walk.

\begin{proposition}\label{prop-recurr}
    There exist constants $\Delta,\delta>0$ so that for $n\geq0$ and $0<\rho\leq1$,
    $$\zeta_n\{\text{inj}\leq\rho\}\ll\rho^{\delta}(1+e^{-\delta n}\text{inj}(x)^{-\Delta}).$$
\end{proposition}

\begin{proof}
    Let $\beta$ be as in \cref{sm}. %, for some $A\geq1$ and $M_0>0$,
    %$$\mu^{*n}*\delta_y(\beta)\leq A e^{-\kappa n}\beta(y)+M_0.$$
    For $0<s\leq1$, Jensen's inequality gives
    \begin{equation}\label{s-contraction}\mu^{*n}*\delta_y(\beta^s)\leq  e^{-\kappa sn}\beta(y)^s+M^s.\end{equation}
    By \cref{h-inj} and Markov's inequality, it suffices to prove
    \begin{equation}\label{main-recurr}
        \zeta_n(\beta^s)\ll e^{-\delta n}\beta(x)^s+1.
    \end{equation}

    Choose $v>0$ small enough in terms of the moment exponents in \mbox{%DIFAUXCMD
\eqref{almost-moment} and \eqref{lambda-finite} }\hskip0pt%DIFAUXCMD
and the growth exponent in \mbox{%DIFAUXCMD
\eqref{growth-rate}}\hskip0pt%DIFAUXCMD
. These estimates and submultiplicativity of $\Vert\Ad(\cdot)\Vert$ give, for some $L>0$ and $m=\lceil\theta n\rceil$,
    $$\zeta_n(\beta^v)+\mu^{*m}*\zeta_{n-m}(\beta^v)
      \ll e^{Lvn}\beta(x)^v.$$
    Let $0<s<v$ with $s\leq1$. Truncating $\beta^s$ smoothly at $\beta=R$, the derivative estimates \eqref{diff-estimate} give a function with $\mathcal S_{\infty,p}$-norm $O(R^s)$. Applying \mbox{%DIFAUXCMD
\eqref{almost-conv} }\hskip0pt%DIFAUXCMD
to this function and estimating both tails by $\beta^s\mathbf1_{\{\beta>R\}}\leq\beta^vR^{s-v}$, we obtain
    \begin{equation}\label{trunca}
        |\zeta_n(\beta^s)-\mu^{*m}*\zeta_{n-m}(\beta^s)|
        \ll R^se^{-cn}+e^{Lvn}\beta(x)^vR^{s-v}.
    \end{equation}
    Now choose
    $$R=e^{(c/v+L)n}\beta(x),\qquad
      0<s<\min\{1,v,cv/(c+Lv)\},\qquad
      \Upsilon=c-s(c/v+L)>0.$$
    We have from \eqref{trunca} that
    \begin{equation}\label{almost-tri}
        |\zeta_n(\beta^s)-\mu^{*m}*\zeta_{n-m}(\beta^s)|
        \ll e^{-\Upsilon n}\beta(x)^s.
    \end{equation}

    On the other hand, the contraction estimate \eqref{s-contraction} shows that
    \begin{equation}\label{conv-tri}
        \mu^{*m}*\zeta_{n-m}(\beta^s)
        \leq e^{-\kappa sm}\zeta_{n-m}(\beta^s)+M^s.
    \end{equation}
    Choose $0<\delta<\min\{\Upsilon,\kappa s\}$. Thus \mbox{%DIFAUXCMD
\eqref{conv-tri} and \eqref{almost-tri} }\hskip0pt%DIFAUXCMD
and the triangle inequality prove \mbox{%DIFAUXCMD
\eqref{main-recurr} }\hskip0pt%DIFAUXCMD
by induction. \end{proof}

\subsection{Effective equidistribution}

As an application of the established effective recurrence, we prove effective equidistribution of $\mu$-almost random walks by combining it with the following equidistribution of the $\mu$-random walk established in \cite{BHZ_2026}.

%Although this equidistribution will not contribute to the proof of the main theorem, it might be of independent interest.

\begin{theorem}[Theorem 1.3 in \cite{BHZ_2026}]
    Let $\mu$ be as defined above. There exists a constant $c=c(\mu,\Lambda)>0$ so that for every $n\geq1$, $f\in B_{\infty,p}(X)$ and $x\in X,$ we have \begin{equation}
        |\mu^{*n}*\delta_x(f)-m_X(f)|\ll\mathcal S_{\infty,p}(f)\inj(x)^{-1}e^{-cn}.\label{conv-equi}
    \end{equation} 
\end{theorem}

We then prove the following corollary.\begin{theorem}
    Let $\mu$ be as defined above, and $\zeta_n=\lambda_n*\delta_x$ be a $\mu$-almost random walk starting at $x\in X.$ Then there exist constants $\Delta,c>0$, depending on $\mu$, $\Lambda$ and the constants in \mbox{%DIFAUXCMD
\cref{almost-def}}\hskip0pt%DIFAUXCMD
, so that for $n\geq1$ and $f\in B_{\infty,p}(X)$, \begin{equation}
        |\zeta_n(f)-m_X(f)|\ll\mathcal S_{\infty,p}(f)\inj(x)^{-\Delta}e^{-cn}.
    \end{equation}
\end{theorem}

\begin{proof}
    Without loss of generality, $m_X(f)=0.$ Then by truncating at $\inj=\rho,$ we have by \eqref{conv-equi} that \begin{equation}\label{eff-tri1}
    |\mu^{*m}*\zeta_{n-m}(f)|\ll\mathcal S_{\infty,p}(f)(e^{-cm}\rho^{-1}+\zeta_{n-m}\{\inj\leq\rho\}).\end{equation} Moreover, \cref{prop-recurr} shows that \begin{equation}\label{eff-tri2}
        \zeta_{n-m}\{\inj\leq\rho\}\ll\rho^{\delta}(1+e^{-\delta(n-m)}\inj(x)^{-\Delta}).
    \end{equation}

    Finally, by setting $m=\lceil\theta n\rceil,\rho=e^{-cm/2},$ where $1>\theta>0$ is as in \cref{almost-def}, we finish the proof by \eqref{eff-tri1}, \eqref{eff-tri2} and \eqref{almost-conv} combined via triangle inequality.
\end{proof}

\section{ Proof of equidistribution theorems}\label{mainproof-sec}
The proof of \cref{main: single} and \cref{main: double} is inspired by the proof of double equidistribution for self-similar measures (given single equidistribution) in \cite[Proposition 6.1]{Benard_He_Zhang_2026}. One of the main tools they use is effective recurrence of random walks  as described in \cite[Section 5]{BHZ_2026}, whereas in this paper it shall be replaced by effective recurrence of almost random walks.

\subsection{Main ideas of proof} \label{idea}For a probability measure $\xi$ on $\R^d,$ we define the process $$\eta_{\xi}^t:=a(t)u(\mathbf s)d\xi(\mathbf s).$$%; see \cref{def-sec} for definitions. 
The key feature of self-similar measure $\sigma$ to establish Khintchine dichotomy via methods in \cite{BHZ_2026} is the following link between $\eta_\sigma^t$ and $\mu$-random walk(\cite[Lemma 8.2]{BHZ_2026}):$$\eta_{\sigma}^t=\int \delta_{k_g}*\eta_\sigma^{tr_g}*\delta_g d\mu^{*n}(g),\forall n\geq 1.$$ In particular, by zooming in to view the measure $g_{*}\sigma$ where $g$ is distributed as $\mu^{*n}$, one can either use high dimension of $\mu^{*n}$ or the existing single equidistribution to show (double) equidistribution.

However, $\nu$, the image of $\sigma,$ does not have such an exact relation. But we can still approximate such link via two steps. 

First, $\eta_\nu^t$ is close to \begin{equation}\label{close-to}\eta_\nu^{t,n}:=\int\delta_{k'_g}*\eta_\sigma^{tr_g}*\delta_g d\lambda_n(g)\end{equation} where $\lambda_n*\delta_x,x\in X$ form a $\mu$-almost random walk. Such closeness will be shown in \eqref{first-tri}, while the corresponding almost random walk will be discussed in \cref{x-phi-verify}.

Second, by \cref{arw-sec}, such almost random walks $(\lambda_n*\delta_x)_{n\geq1}$ have fruitful structures, similar to $\mu^{*n}*\delta_x,$ to effectively recur or even equidistribute. 

Note that using this approximation, one can follow the route of dimension bootstrap as in \cite{BHZ_2026} to reach the main theorem of this paper. For the sake of simplicity, we make this procedure implicit by utilizing the equidistribution result proven in \cite{BHZ_2026}.

\subsection{An example of almost random walks}
\label{x-phi-verify}
In this subsection we provide an example of almost random walks related to $\nu=h_*\sigma$. In particular, for $x\in X$, we define\begin{equation}\label{def-xphi}x_\phi:=k_{Dh(\mathbf b_\phi)O_\phi}^{-1}a(r_\phi^{-1})u(h(\mathbf b_\phi))x.\end{equation}We prove the following lemma,\begin{lemma} \label{xphi-almost}The laws
$$\zeta_n=\int\delta_{x_\phi}\,d\lambda^{*n}(\phi)$$
form a $\mu$-almost random walk starting from $y=k_{Dh(\mathbf0)}^{-1}u(h(\mathbf0))x.$\end{lemma}
\begin{proof}
Write $D_\phi=Dh(\mathbf b_\phi)$. Then $x_\phi=Q_\phi y$, where
$$Q_\phi=k_{D_\phi O_\phi}^{-1}a(r_\phi^{-1})u(h(\mathbf b_\phi)-h(\mathbf0))k_{Dh(\mathbf0)}.$$
The finite moment condition follows from \mbox{%DIFAUXCMD
\eqref{linear-sta},\eqref{lambda-finite},\eqref{bdd-assumption} and \eqref{uniform-holder}}\hskip0pt%DIFAUXCMD
. More explicitly, for all sufficiently small $s>0$,
\begin{equation}\label{translation-moment}
\sup_{j\geq0}\int\Vert\mathbf b_\phi\Vert^s\,d\lambda^{*j}(\phi)<\infty,
\qquad
\int r_\phi^{\pm s}\,d\lambda^{*j}(\phi)=\left(\int r_\phi^{\pm s}\,d\lambda(\phi)\right)^j<\infty,\int r_\phi^{ s}\,d\lambda(\phi)<1.
\end{equation}
We refer the details on how to get bounds for $\mathbf b_\phi$ as above to \cite[Section 3]{BHZ_2026}. Since
$$\Vert h(\mathbf b_\phi)-h(\mathbf0)\Vert\ll\Vert\mathbf b_\phi\Vert$$
by \mbox{%DIFAUXCMD
\eqref{bdd-assumption}}\hskip0pt%DIFAUXCMD
, and $k_{D_\phi O_\phi}$ and its inverse are uniformly bounded, \mbox{%DIFAUXCMD
\eqref{translation-moment} }\hskip0pt%DIFAUXCMD
and Cauchy--Schwarz give
$$\int\Vert\Ad(Q_\phi)\Vert^\gamma\,d\lambda^{*n}(\phi)\leq e^{Kn}$$
for some $\gamma,K>0$ and every $n\geq1$.

The rest of this proof will be devoted to proving the close-to-convolution property.

Note that $x_{\psi\phi}=z_{\psi,\phi}x_\psi$, where the step from $x_\psi$ to $x_{\psi\phi}$ is
\begin{equation*}
z_{\psi,\phi}=k_{D_{\psi\phi}O_{\psi\phi}}^{-1}a(r_{\psi\phi}^{-1})
u(h(\mathbf b_{\psi\phi})-h(\mathbf b_\psi))a(r_\psi)k_{D_\psi O_\psi}.
\end{equation*}
In the spirit of Hölder regularity assumption, we set the remainder
$$E_{\psi,\phi}=h(\mathbf b_{\psi\phi})-h(\mathbf b_\psi)-r_\psi D_\psi O_\psi\mathbf b_\phi.$$
Since $\mathbf b_{\psi\phi}=\mathbf b_\psi+r_\psi O_\psi\mathbf b_\phi$ and $O_{\psi\phi}=O_\psi O_\phi$, we have that the step $z_{\psi,\phi}$ is close to a random step $g_\phi$ in the following sense
\begin{equation}\label{cocycle-error}
\begin{split}
z_{\psi,\phi}g_\phi^{-1}
&=k_{D_{\psi\phi}O_{\psi\phi}}^{-1}
u\left(\frac{E_{\psi,\phi}}{r_\psi r_\phi}\right)k_{D_\psi O_{\psi\phi}},\\
\Vert z_{\psi,\phi}g_\phi^{-1}-I\Vert
&\leq \left\Vert k_{D_{\psi\phi}O_{\psi\phi}}^{-1}
k_{D_\psi O_{\psi\phi}}-I \right\Vert+\left\Vert k_{D_{\psi\phi}O_{\psi\phi}}^{-1}
\left(u\left(\frac{E_{\psi,\phi}}{r_\psi r_\phi}\right)-I\right)k_{D_\psi O_{\psi\phi}} \right\Vert  \\&\overset{\eqref{uniform-holder}}{\ll} r_\psi^\alpha\left(\Vert\mathbf b_\phi\Vert^\alpha+
\frac{\Vert\mathbf b_\phi\Vert^{1+\alpha}}{r_\phi}\right).
\end{split}
\end{equation}
%Here the second line uses both bounds in \mbox{%DIFAUXCMD
%\cref{uniform-holder}}\hskip0pt%DIFAUXCMD
%; the first term accounts for the change in $Dh$.

Choose $0<q\leq1$ sufficiently small that \mbox{%DIFAUXCMD
\eqref{translation-moment} }\hskip0pt%DIFAUXCMD
holds with $s=2(1+\alpha)q$ and
\begin{equation}\label{decay}
a:=\int r_\phi^{\alpha q}\,d\lambda(\phi)<1,\qquad
B:=\max\left\{1,\left(\int r_\phi^{-2q}\,d\lambda(\phi)\right)^{1/2}\right\}<\infty.
\end{equation}
By Cauchy--Schwarz and \mbox{%DIFAUXCMD
\eqref{translation-moment}}\hskip0pt%DIFAUXCMD
,
$$\int\left(\Vert\mathbf b_\phi\Vert^{\alpha q}+\Vert\mathbf b_\phi\Vert^{(1+\alpha)q}r_\phi^{-q}\right)\,d\lambda^{*m}(\phi)\ll B^m.$$
 For $f\in B_{\infty,p}(X)$(which in particular bounded $C^1$), $|f(gy)-f(y)|\ll\mathcal S_{\infty,p}(f)\min\{1,\Vert g-I\Vert\}$. Applying this with \mbox{%DIFAUXCMD
\eqref{cocycle-error} }\hskip0pt%DIFAUXCMD
and independence between $\phi\sim\lambda^{*m}$ and $\psi\sim\lambda^{*(n-m)}$, we have
\begin{equation}\begin{split}
&|\zeta_n(f)-\mu^{*m}*\zeta_{n-m}(f)|=\left|\int (f(x_{\psi\phi})-f(g_\phi x_\psi))d\lambda^{*m}(\phi)d\lambda^{*(n-m)}(\psi)\right|\\
&\ll\mathcal S_{\infty,p}(f)
\int r_\psi^{\alpha q}\left(\Vert\mathbf b_\phi\Vert^{\alpha q}
+\frac{\Vert\mathbf b_\phi\Vert^{(1+\alpha)q}}{r_\phi^q}\right)
\,d\lambda^{*m}(\phi)\,d\lambda^{*(n-m)}(\psi)\\
&\ll\mathcal S_{\infty,p}(f)a^{n-m}B^m.\label{almost-walk}
\end{split}\end{equation}
%We used that $g_{\psi\phi}=g_\phi g_\psi$, so the image of $\lambda^{*m}$ under $\phi\mapsto g_\phi$ is $\mu^{*m}$.

Finally, set $m=\lceil\theta n\rceil$ with $\theta>0$ sufficiently small that
$$(1-\theta)\log a+\theta\log B<0.$$
Then \mbox{%DIFAUXCMD
\eqref{almost-walk} and \eqref{decay} }\hskip0pt%DIFAUXCMD
show that $\zeta_n$ is a $\mu$-almost random walk.\end{proof}

\subsection{Equidistribution}\label{single-verify}

In this subsection, we aim to prove \cref{main: single}.

%\subsection{proof of \cref{main: single}}

We first recall the stationarity of $\nu$ by iterating \eqref{linear-sta}  $n$ times,
\begin{equation}
    \label{Rd-sta}\nu=\int(h\circ\phi)_*\sigma\,d\lambda^{*n}(\phi).
\end{equation}

Denote $$\mathcal E_{n,M}:=\{\phi\in\text{Sim}(\R^d):|nl+\ln r_\phi|\leq \frac{nl}{M}\},$$ then by large deviation principles and \eqref{contra}, there exists $\epsilon=\epsilon(\lambda,M)>0$ so that for $n\gg_M1,$ \begin{equation}\label{l-deviation}
    \lambda^{*n}(\mathcal E_{n,M})\geq 1-e^{-\epsilon n}.
\end{equation}

Then denote$$I(n,f,t,x):=\int_{\text{Sim}(\R^d)}\left|\int f(a(t)u(\mathbf s)x)d((h\circ\phi)_*\sigma)(\mathbf s)-m_X(f)\right|d\lambda^{*n}(\phi).$$ Note that by \eqref{Rd-sta} and triangle inequality,\begin{equation}\label{I1-tri}
    \left|\int f(a(t)u(\mathbf s)x)d(h_*\sigma)(\mathbf s)-m_X(f)\right|\leq I(n,f,t,x).
\end{equation} By \eqref{l-deviation} we can focus on $\mathcal E_{n,M}$ and have\begin{equation}\begin{split}\label{1st-I1}
    I(n,f,t,x){\leq}\int_{\mathcal E_{n,M}}\left|\int f(a(t)u(h\circ\phi(\mathbf s))x)d\sigma(\mathbf s)-m_X(f)\right|d\lambda^{*n}(\phi)+O(\mathcal S_{\infty,p}(f)e^{-\epsilon n}).
\end{split}
\end{equation}

Now fix some $\mathbf s_0\in\supp\sigma.$ Without loss of generality, $\mathbf s_0=\mathbf0$ (by translating coordinates and conjugating the similarities accordingly). Stationarity implies that $\phi(\supp\sigma)\subseteq \supp\sigma$ for $\lambda^{*n}$-almost every $\phi$, so the following evaluations of $h$ and $Dh$ are defined. Then we use \mbox{%DIFAUXCMD
\eqref{uniform-holder} }\hskip0pt%DIFAUXCMD
to linearize and see that for $\lambda^{*n}$-almost every $\phi\in\mathcal E_{n,M}$ and $\mathbf s\in \supp\sigma$ with $\Vert \mathbf s\Vert\leq e^{\frac{nl}{M}},$ we have that \begin{equation}\begin{split}\label{holder}f[a(t)u(h\circ\phi(\mathbf s))x]&=f[a(t)u(h\circ\phi(\mathbf0)+Dh(\mathbf b_\phi)r_\phi O_\phi\mathbf s+O(e^{-(1-\frac{2}{M})(1+\alpha)nl}))x]\\&=f[k_{Dh(\mathbf b_\phi)O_\phi} a(tr_\phi)u(\mathbf s)x_\phi]+O(\mathcal S_{\infty,p}(f)te^{-(1-\frac{2}{M})(1+\alpha)nl}).\end{split}\end{equation}
where $x_\phi$ is as in \eqref{def-xphi}. %$$$$x_\phi=k_{Dh(\mathbf b_\phi)O_\phi}^{-1}a(r_\phi^{-1})u(h(\mathbf b_\phi))x.$$

By \cite{BHZ_2026}, $\sigma$ has a polynomial tail:\begin{lemma}[Lemma 3.1 in \cite{BHZ_2026}]\label{poly-tail}
    There exists $\gamma>0$ so that $$\int\Vert\mathbf s\Vert^{\gamma}d\sigma(\mathbf s)<\infty.$$
\end{lemma}

Then by \eqref{holder} and \cref{poly-tail}, we have that for $\lambda^{*n}$-almost every $\phi\in\mathcal E_{n,M},$ \begin{equation}\label{first-tri}\begin{split}
    &\left|\int \left\{f[a(t)u(h\circ\phi(\mathbf s))x]-f[k_{Dh(\mathbf b_\phi)O_\phi} a(tr_\phi)u(\mathbf s)x_\phi]\right\}d\sigma(\mathbf s)\right|\\&\quad\ll \mathcal S_{\infty,p}(f)(e^{-\frac{nl\gamma}{M}}+te^{-(1-\frac{2}{M})(1+\alpha)nl}).
\end{split}\end{equation}

 By \eqref{1st-I1} and \eqref{first-tri}, we can now change the base point from $x$ to $x_\phi$ and approximately rescale $(h\circ\phi)_*\sigma$ to $\sigma;$ in particular, we have \begin{equation}\label{2nd-I1}
    \begin{split}
        I(n,f,t,x)&{\leq}\int_{\mathcal E_{n,M}}\left|\int f(k_{Dh(\mathbf b_\phi)O_\phi}a(tr_\phi)u(\mathbf s)x_\phi)d\sigma(\mathbf s)-m_X(f)\right|d\lambda^{*n}(\phi)\\&+O(\mathcal S_{\infty,p}(f)(e^{-\epsilon n}+e^{-\frac{nl\gamma}{M}}+te^{-(1-\frac{2}{M})(1+\alpha)nl})).
    \end{split}
\end{equation}

On the other hand, by boundedness of $k_{Dh(\mathbf b_\phi)O_\phi}$ and its inverse from \eqref{bdd-assumption}, $G$-invariance of Haar measure, and \cref{Rd-equi}, we have, provided $tr_\phi>1$,\begin{equation}\label{second-tri}
    \left|\int f[k_{Dh(\mathbf b_\phi)O_\phi}a(tr_\phi)u(\mathbf s)x_\phi]d\sigma(\mathbf s)-m_X(f)\right|\ll \mathcal S_{\infty,p}(f)\text{inj}(x_\phi)^{-1}t^{-c}e^{c(1+\frac{1}{M})nl},\phi\in\mathcal E_{n,M}
\end{equation}

Then the thing left over is to show that for all $\phi$ except a set of small measure with respect to $\lambda^{*n},$ $\text{inj}(x_\phi)$ is not small. We use the following effective recurrence for $x_\phi$ from \cref{prop-recurr} and the fact that the laws of $x_\phi$, for $\phi\sim\lambda^{*n}$, form a $\mu$-almost random walk starting from $y:=k_{Dh(\mathbf0)}^{-1}u(h(\mathbf0))x$, as proved in \mbox{%DIFAUXCMD
\cref{xphi-almost}}\hskip0pt%DIFAUXCMD
, where $\inj(y)\asymp_h\inj(x)$.\begin{proposition}\label{nonlinear-recurr}There exist constants $\Delta,\delta>0$ so that for $n\geq0$ and $0<\rho\leq1$,
    $$\lambda^{*n}(\phi:\text{inj}(x_\phi)\leq\rho)\ll\rho^{\delta}(1+e^{-\delta n}\text{inj}(x)^{-\Delta}).$$
\end{proposition}

Finally, applying \eqref{second-tri} and \cref{nonlinear-recurr} with $\rho=e^{-\frac{cnl}{M}}$ to \eqref{2nd-I1}, we have\begin{equation}\begin{split}
\label{Rd-main bd}    I(n,f,t,x)\overset{}{\ll} \mathcal S_{\infty,p}(f)(e^{-\epsilon n}+\text{inj}(x)^{-\Delta}e^{-\frac{\delta cnl}{M}}+e^{-\frac{nl\gamma}{M}}+te^{-(1-\frac{2}{M})(1+\alpha)nl}+t^{-c}e^{c(1+\frac{2}{M})nl}).\end{split}
\end{equation} %On $\{\inj(x_\phi)\leq\rho\}$ we use the trivial bound $2\Vert f\Vert_\infty$. The same estimates bound the average of the absolute conditional errors in the first line.

Now combining \eqref{I1-tri} and \eqref{Rd-main bd}, the proof of \cref{main: single} finishes by choosing $M$ large enough so that $$(1-\frac{2}{M})(1+\alpha)>1+\frac{2}{M},$$ then choosing $$\ln t=\beta nl+O(1),\beta\in(1+\frac{2}{M},(1-\frac{2}{M})(1+\alpha)).$$ %This choice also ensures $tr_\phi>1$ on $\mathcal E_{n,M}$ for all sufficiently large $t$; bounded $t$ are absorbed by increasing the constant.

\subsection{Double equidistribution}

In this subsection we prove \cref{main: double}.

     To lighten notation, we write $\mathcal{S} = \mathcal{S}_{\infty,p}$ and $\mathcal{S}(f_1, f_2) = \mathcal{S}(f_1)\mathcal{S}(f_2)$. Noting the relation
\[
\Delta_{f_1,f_2}^\nu(t_1,t_2) \le \Delta_{f_1,f_2-m_X(f_2)}^\nu(t_1,t_2) + \left| \int_{\R^d} f_1(a(t_1)u(\mathbf s)x_0) \, \mathrm{d}\nu(\mathbf s) - m_X(f_1) \right| |m_X(f_2)|
\]
and that \cref{main: single} provides us with a constant $c' = c'(h,\Lambda, \sigma) > 0$ such that
\[
\left| \int_{\R^d} f_1(a(t_1)u(\mathbf s)x_0) \, \mathrm{d}\nu(\mathbf s) - m_X(f_1) \right| \ll \mathcal{S}(f_1)t_1^{-c'},
\]
we can reduce to the case where $m_X(f_2) = 0$.

Thus, we are left to bound the integral
\[
J := \int_{\R^d} f_1(a(t_1)u(\mathbf s)x_0) f_2(a(t_2)u(\mathbf s)x_0) \, \mathrm{d}\nu(\mathbf s)
\]
by a quantity of the form $O_\eta(\mathcal{S}(f_1, f_2)t_2^{-\kappa})$ where $\kappa = \kappa(\Lambda, h,\sigma, \eta) > 0$.

Again by stationarity \eqref{Rd-sta} and large deviation principles \eqref{l-deviation}, we have\begin{equation}\label{1st-I2}
    \begin{split}
        |J|{\leq}\left|\int_{\mathcal E_{n,M}}\int_{\R^d} f_1(a(t_1)u(\mathbf s)x_0)f_2(a(t_2)u(\mathbf s)x_0) \, \mathrm{d}(h\circ\phi)_*\sigma(\mathbf s)d\lambda^{*n}(\phi)\right|+O(\mathcal S(f_1,f_2)e^{-\epsilon n}).
    \end{split}
\end{equation}

Now for $\lambda^{*n}$-almost every $\phi\in\mathcal E_{n,M}$ and $\mathbf s\in \supp\sigma$ with $\Vert\mathbf s\Vert\leq e^{\frac{nl}{M}}$ we have that \begin{equation}
    \label{bdd-Taylor}
|f_1(a(t_1)u(h\circ\phi(\mathbf s))x_0)-f_1(a(t_1)u(h(\mathbf b_\phi))x_0)|\ll \mathcal S(f_1) t_1e^{-(1-\frac{2}{M})nl},\end{equation} where \mbox{%DIFAUXCMD
\eqref{bdd-assumption} and \eqref{uniform-holder} }\hskip0pt%DIFAUXCMD
give $\Vert h(\phi(\mathbf s))-h(\mathbf b_\phi)\Vert\ll r_\phi\Vert\mathbf s\Vert$ on this range.

Thus by \eqref{1st-I2}, \eqref{bdd-Taylor} and polynomial tail of $\sigma$ \cref{poly-tail}, we have \begin{equation}\begin{split}|J|&{\leq} \mathcal S(f_1)\int_{\text{Sim}(\R^d)}\left|\int_{\R^d} f_2(a(t_2)u(\mathbf s)x_0) \, \mathrm{d}(h\circ\phi)_*\sigma(\mathbf s)\right|d\lambda^{*n}(\phi)\\&+O(\mathcal S(f_1,f_2)(e^{-\epsilon n}+e^{-\frac{nl\gamma}{M}}+t_1e^{-(1-\frac{2}{M})nl}))\\&=\mathcal S(f_1)I(n,f_2,t_2,x_0)+O(\mathcal S(f_1,f_2)(e^{-\epsilon n}+e^{-\frac{nl\gamma}{M}}+t_1e^{-(1-\frac{2}{M})nl})).\label{2nd-I2}\end{split}\end{equation}

Finally, by \eqref{2nd-I2} and control for the single equidistribution \eqref{Rd-main bd}, we have \begin{equation}
    \begin{split}
        |J|&{\ll}\mathcal S(f_1,f_2)\bigl(e^{-\epsilon n}+\text{inj}(x_0)^{-\Delta}e^{-\frac{\delta cnl}{M}}+e^{-\frac{nl\gamma}{M}}+t_1e^{-(1-\frac{2}{M})nl}\\&\qquad{}+t_2e^{-(1-\frac{2}{M})(1+\alpha)nl}+t_2^{-c}e^{c(1+\frac{2}{M})nl}\bigr).
    \end{split}
\end{equation}

%where the last inequality comes from \mbox{%DIFAUXCMD
%\cref{Rd-main bd}}\hskip0pt%DIFAUXCMD
%, with $x=x_0$.
%

Now we finish the proof of \cref{main: double} by choosing $$nl=\omega \ln t_2+(1-\omega)\ln t_1+O(1),\frac{1}{\omega}\in(1+\frac{2}{M},(1-\frac{2}{M})(1+\alpha)).$$

%$M$ large enough so that
%$$\min\{(1-\tfrac{2}{M})(1+\eta),(1-\tfrac{2}{M})(1+\alpha)\}>1+\tfrac{2}{M},$$
%then choosing
%$$q\in\left(\max\left\{\frac{1}{(1-\frac{2}{M})(1+\eta)},\frac{1}{(1-\frac{2}{M})(1+\alpha)}\right\},\frac{1}{1+\frac{2}{M}}\right),\qquad
%n=\left\lfloor\frac{q\ln t_2}{l}\right\rfloor.$$
%Indeed, $t_1\leq t_2^{1/(1+\eta)}$, so every term above is bounded by a negative power of $t_2$. The bounded range of $t_2$ is absorbed by increasing the constant.
%

\section{Future directions}\label{future-sec}

One direction to extend the current setting is to unfold the dimension bootstrap hidden in \cref{Rd-equi}. In particular, as described at the end of \cref{idea}, one may go beyond the images of self-similar measures by revisiting the full path in \cite{BHZ_2026} and making suitable approximation assumptions. So we make the following conjecture:

\begin{conjecture}
    Suppose a probablity measure $\xi$ on $\R^d$ satisfies that $\eta_{\xi}^t$ is `close to' some $\eta_\xi^{t,n}$ as described in \eqref{close-to}. Then such $\xi$ would satisfy (double) equidistribution and thus Khintchine dichotomy.
\end{conjecture}

On the other hand, we have established Khintchine dichotomy for nonlinear repellers conjugate to self-similar set, but the family of repellers is much more fruitful.
Indeed, the things we have so far implicitly rely on the applications of projection theorems as established in \cite{BHZ_2026,Benard_He_Zhang_2026}. More general conditions can be explored via making dimension bootstrap explicit. Nonlinear repellers seem to be a reasonable substitute with certain structure to bootstrap(e.g. the recent work \cite{LiPanTongXu2026}), thus leading to the following interesting question\begin{problem}
    What is the genuine condition needed for measures from repellers on $\R^d$ to establish equidistribution as in \cref{main: single}? What about Khintchine dichotomy?
\end{problem}

It is worthwile to remark that an explicit counterexample for the problem above is at hand: the repellers defined by Gauss map $x\to\{\frac{1}{x}\}$ on finitely many intervals of the form $[1/n,1/(n-1)]$, consist of only badly approximable numbers with bounded continued fraction digits.

Another route would be to extend the setting to general nondegenerate maps $h:\R^d\to\R^k.$ In particular, when $d<k,$ this is the problem of Khintchine dichotomy on nondegenerate manifolds(see e.g. \cite{BY,BeresnevichDattaYang2026} for a dynamics proof over Lebesgue measure):\begin{problem}
    Is the Khintchine dichotomy true for self-similar measures on general nondegenerate manifolds? If not, what condition would suffice?
\end{problem}

Note that equidistribution of expanding translates has been proved in \cite{ShahYang2023} for suitable sufficiently smooth nondegenerate manifolds and absolutely continuous parameter measures, using shift-invariance of Lebesgue measure and Ratner's theorem. However, this route seems difficult for self-similar measures which doesn't necessarily have shift-invariance. Possibly, a new framework for upgrading equidistribution on certain blocks of the matrices to the full group via nondegeneracy of $h$ is needed.

\bibliographystyle{abbrv}
\bibliography{image_Khintchine.bib}

@preamble{ " \newcommand{\noop}[1]{} " }

@article{Benard_He_Zhang_2026,
  author = {B{\'e}nard, Timoth{\'e}e and He, Weikun and Zhang, Han},
  title = {{Khintchine} dichotomy for self-similar measures},
  journal = {Journal of the American Mathematical Society},
  volume = {39},
  number = {3},
  pages = {587--623},
  year = {2026},
  doi = {10.1090/jams/1070},
  url = {https://doi.org/10.1090/jams/1070}
}

@article {C,
    AUTHOR = {Cheung, Yitwah},
     TITLE = {Hausdorff dimension of the set of singular pairs},
   JOURNAL = {Ann. of Math. (2)},
  FJOURNAL = {Annals of Mathematics. Second Series},
    VOLUME = {173},
      YEAR = {2011},
    NUMBER = {1},
     PAGES = {127--167},
      ISSN = {0003-486X},
   MRCLASS = {22E40 (11K55 37A15 37D40)},
  MRNUMBER = {2753601},
MRREVIEWER = {Thomas Ward},
       DOI = {10.4007/annals.2011.173.1.4},
       URL = {https://doi-org.proxy.lib.umich.edu/10.4007/annals.2011.173.1.4},
}

@article {KN,
    AUTHOR = {Kim, Dong Han and Nakada, Hitoshi},
     TITLE = {Metric inhomogeneous {D}iophantine approximation on the field
              of formal {L}aurent series},
   JOURNAL = {Acta Arith.},
  FJOURNAL = {Acta Arithmetica},
    VOLUME = {150},
      YEAR = {2011},
    NUMBER = {2},
     PAGES = {129--142},
      ISSN = {0065-1036},
   MRCLASS = {11J61 (11J83 11K60)},
  MRNUMBER = {2836382},
MRREVIEWER = {Radhakrishnan Nair},
       DOI = {10.4064/aa150-2-3},
       URL = {https://doi-org.proxy.lib.umich.edu/10.4064/aa150-2-3},
}

@article {BY,
    AUTHOR = {Beresnevich, Victor and Yang, Lei},
     TITLE = {Khintchine's theorem and {D}iophantine approximation on
              manifolds},
   JOURNAL = {Acta Math.},
  FJOURNAL = {Acta Mathematica},
    VOLUME = {231},
      YEAR = {2023},
    NUMBER = {1},
     PAGES = {1--30},
      ISSN = {0001-5962,1871-2509},
   MRCLASS = {11J83 (11K55 11K60 57N35)},
  MRNUMBER = {4652409},
MRREVIEWER = {Hemangi\ Madhusudan\ Shah},
}

@article {d,
    AUTHOR = {de Mathan, Bernard},
     TITLE = {Approximations diophantiennes dans un corps local},
   JOURNAL = {Bull. Soc. Math. France Suppl. M\'{e}m.},
  FJOURNAL = {Bulletin de la Soci\'{e}t\'{e} Math\'{e}matique de France. Supplement.
              M\'{e}moire},
    VOLUME = {21},
      YEAR = {1970},
     PAGES = {93},
      ISSN = {0583-8665},
   MRCLASS = {10.34},
  MRNUMBER = {274396},
MRREVIEWER = {W. J. LeVeque},
}

@article{KLW,
  author = {Kleinbock, Dmitry and Lindenstrauss, Elon and Weiss, Barak},
  title = {On fractal measures and {Diophantine} approximation},
  journal = {Selecta Mathematica. New Series},
  volume = {10},
  number = {4},
  pages = {479--523},
  year = {2004},
  doi = {10.1007/s00029-004-0378-2},
  url = {https://doi.org/10.1007/s00029-004-0378-2}
}

@incollection {R,
    AUTHOR = {Qu\^{e}me, Roland},
     TITLE = {On {D}iophantine approximation by algebraic numbers of a given
              number field: a new generalization of {D}irichlet
              approximation theorem},
      NOTE = {Journ\'{e}es Arithm\'{e}tiques, 1989 (Luminy, 1989)},
   JOURNAL = {Ast\'{e}risque},
  FJOURNAL = {Ast\'{e}risque},
    NUMBER = {198-200},
      YEAR = {1991},
     PAGES = {273--283 (1992)},
      ISSN = {0303-1179},
   MRCLASS = {11J17 (11J68)},
  MRNUMBER = {1144331},
MRREVIEWER = {S. V. Kotov},
}

@article {M,
    author = {Mahler, Kurt},
    title = {An analogue to Minkowski's geometry of numbers in a field of series},
    journal = {Ann. of Math},
    volume = {(2) 42},
    year = {1941},
    pages = {488?522},

}

@article {S,
    author = {W. M. Schmidt},
    title = {Simultaneous approximation to algebraic numbers by elements of a number field},
    journal = {Monatsh. Math.},
    volume = {79},
    year = {1975},
    pages = {55-66},

}

@article{Bougerol_Picard_1992,
  author = {Bougerol, Philippe and Picard, Nico},
  title = {Strict stationarity of generalized autoregressive processes},
  journal = {The Annals of Probability},
  volume = {20},
  number = {4},
  pages = {1714--1730},
  year = {1992},
  doi = {10.1214/aop/1176989526},
  url = {https://doi.org/10.1214/aop/1176989526}
}

@misc{BHZ_2026,
  author = {B{\'e}nard, Timoth{\'e}e and He, Weikun and Zhang, Han},
  title = {{Khintchine} dichotomy and {Schmidt} estimates for self-similar measures on {$\mathbb{R}^d$}},
  year = {2025},
  eprint = {2508.09076},
  archivePrefix = {arXiv},
  primaryClass = {math.DS},
  howpublished = {\url{https://arxiv.org/abs/2508.09076v2}},
  note = {Preprint, arXiv:2508.09076v2, revised 21 February 2026},
  doi = {10.48550/arXiv.2508.09076}
}

@article{Prohaska_Sert_Shi_2023,
  author = {Prohaska, Roland and Sert, {\c{C}}a{\u{g}}r{\i} and Shi, Ronggang},
  title = {Expanding measures: Random walks and rigidity on homogeneous spaces},
  journal = {Forum of Mathematics, Sigma},
  volume = {11},
  pages = {e59},
  year = {2023},
  doi = {10.1017/fms.2023.56},
  url = {https://doi.org/10.1017/fms.2023.56}
}

@article{Mahler_1984,
  author = {Mahler, Kurt},
  title = {Some suggestions for further research},
  journal = {Bulletin of the Australian Mathematical Society},
  volume = {29},
  number = {1},
  pages = {101--108},
  year = {1984},
  doi = {10.1017/S0004972700021316},
  url = {https://doi.org/10.1017/S0004972700021316}
}

@article{Pollington_Velani_2005,
  author = {Pollington, Andrew and Velani, Sanju L.},
  title = {Metric {Diophantine} approximation and `absolutely friendly' measures},
  journal = {Selecta Mathematica. New Series},
  volume = {11},
  number = {2},
  pages = {297--307},
  year = {2005},
  doi = {10.1007/s00029-005-0007-8},
  url = {https://doi.org/10.1007/s00029-005-0007-8}
}

@article{Simmons_Weiss_2019,
  author = {Simmons, David and Weiss, Barak},
  title = {Random walks on homogeneous spaces and {Diophantine} approximation on fractals},
  journal = {Inventiones Mathematicae},
  volume = {216},
  number = {2},
  pages = {337--394},
  year = {2019},
  doi = {10.1007/s00222-019-00856-7},
  url = {https://doi.org/10.1007/s00222-019-00856-7}
}

@article{Khalil_Luethi_2023,
  author = {Khalil, Osama and Luethi, Manuel},
  title = {Random walks, spectral gaps, and {Khintchine}'s theorem on fractals},
  journal = {Inventiones Mathematicae},
  volume = {232},
  number = {2},
  pages = {713--831},
  year = {2023},
  doi = {10.1007/s00222-022-01171-4},
  url = {https://doi.org/10.1007/s00222-022-01171-4}
}

@misc{Yu_2021,
  author = {Yu, Han},
  title = {Rational points near self-similar sets},
  year = {2021},
  eprint = {2101.05910},
  archivePrefix = {arXiv},
  primaryClass = {math.NT},
  howpublished = {\url{https://arxiv.org/abs/2101.05910}},
  note = {Preprint, arXiv:2101.05910},
  doi = {10.48550/arXiv.2101.05910}
}

@misc{Datta_Jana_2024, title={On Fourier asymptotics and effective equidistribution},eprint = {2407.11961}, archivePrefix = {arXiv},
  primaryClass = {math.DS},
  howpublished = {\url{https://arxiv.org/abs/2407.11961}}, note = {Preprint, arxiv: 2407.11961v3, revised 12 July 2026}, author={Datta, Shreyasi and Jana, Subhajit}, year={2024}, month={Jul}}

@article{AlgomChangWuWu2025,
  author  = {Algom, Amir and Chang, Yuanyang and Wu, Meng and Wu, Yu-Liang},
  title   = {{Van der Corput} and metric theorems for geometric
             progressions for self-similar measures},
  journal = {Mathematische Annalen},
  volume  = {393},
  pages   = {183--214},
  year    = {2025},
  doi     = {10.1007/s00208-025-03233-3}
}

@article{BakerBanaji2025,
  author  = {Baker, Simon and Banaji, Amlan},
  title   = {Polynomial {Fourier} decay for fractal measures
             and their pushforwards},
  journal = {Mathematische Annalen},
  volume  = {392},
  pages   = {209--261},
  year    = {2025},
  doi     = {10.1007/s00208-025-03091-z}
}

@incollection{Kaufman1984,
  author    = {Kaufman, Robert},
  title     = {On {Bernoulli} convolutions},
  booktitle = {Conference in Modern Analysis and Probability
               (New Haven, Conn., 1982)},
  series    = {Contemporary Mathematics},
  volume    = {26},
  pages     = {217--222},
  publisher = {American Mathematical Society},
  address   = {Providence, RI},
  year      = {1984},
  doi       = {10.1090/conm/026/737403}
}

@misc{BanajiYu2026Qualitative,
  author        = {Banaji, Amlan and Yu, Han},
  title         = {{Fourier} transform of nonlinear images of
                   self-similar measures: qualitative aspects},
  year          = {2026},
  eprint        = {2606.09743},
  archivePrefix = {arXiv},
  primaryClass  = {math.DS},
  doi           = {10.48550/arXiv.2606.09743},
  howpublished = {\url{https://arxiv.org/abs/2606.09743}}
}

@misc{LiPanTongXu2026,
  author        = {Li, Jialun and Pan, Wenyu and Tong, Yao and Xu, Disheng},
  title         = {Dimensions of surface repellers and attractors
                   of non-linear planar {IFS}s},
  year          = {2026},
  eprint        = {2608.30744},
  archivePrefix = {arXiv},
  primaryClass  = {math.DS},
  howpublished  = {\url{https://arxiv.org/abs/2608.30744}}
}

@misc{BeresnevichDattaYang2026,
  author        = {Beresnevich, Victor and Datta, Shreyasi and Yang, Lei},
  title         = {Weighted {Diophantine} approximation on manifolds},
  year          = {2026},
  eprint        = {2602.11045},
  archivePrefix = {arXiv},
  primaryClass  = {math.NT},
  howpublished  = {\url{https://arxiv.org/abs/2602.11045}}
}

@article{ShahYang2023,
  author       = {Shah, Nimish A. and Yang, Pengyu},
  title        = {Limit distributions of expanding translates of shrinking
                  submanifolds and non-improvability of {Dirichlet}'s
                  approximation theorem},
  journal      = {Journal of Modern Dynamics},
  volume       = {19},
  pages        = {947--965},
  year         = {2023},
  doi          = {10.3934/jmd.2023027},
  howpublished = {\url{https://doi.org/10.3934/jmd.2023027}}
}

@article{Khintchine1926,
  author  = {Khintchine, A.},
  title   = {Zur metrischen {Theorie} der diophantischen {Approximationen}},
  journal = {Mathematische Zeitschrift},
  volume  = {24},
  year    = {1926},
  pages   = {706--714},
  doi     = {10.1007/BF01216806}
}

@article{BenoistQuint2013III,
  author  = {Benoist, Yves and Quint, Jean-Fran{\c{c}}ois},
  title   = {Stationary measures and invariant subsets of homogeneous spaces ({III})},
  journal = {Annals of Mathematics},
  volume  = {178},
  number  = {3},
  year    = {2013},
  pages   = {1017--1059},
  doi     = {10.4007/annals.2013.178.3.5}
}

@article{Weiss2001,
  author  = {Weiss, Barak},
  title   = {Almost no points on a {Cantor} set are very well approximable},
  journal = {Proceedings of the Royal Society A: Mathematical, Physical and Engineering Sciences},
  volume  = {457},
  number  = {2008},
  year    = {2001},
  pages   = {949--952},
  doi     = {10.1098/rspa.2000.0700}
}

@article{BenoistQuint2013II,
  author  = {Benoist, Yves and Quint, Jean-Fran{\c{c}}ois},
  title   = {Stationary measures and invariant subsets of homogeneous spaces ({II})},
  journal = {Journal of the American Mathematical Society},
  volume  = {26},
  number  = {3},
  year    = {2013},
  pages   = {659--734},
  doi     = {10.1090/S0894-0347-2013-00760-2}
}

@article{Schmidt1960,
  author  = {Schmidt, Wolfgang},
  title   = {A metrical theorem in {Diophantine} approximation},
  journal = {Canadian Journal of Mathematics},
  volume  = {12},
  year    = {1960},
  pages   = {619--631},
  doi     = {10.4153/CJM-1960-056-0}
}

@article{KleinbockShiWeiss2017,
  author  = {Kleinbock, Dmitry and Shi, Ronggang and Weiss, Barak},
  title   = {Pointwise equidistribution with an error rate and with respect to unbounded functions},
  journal = {Mathematische Annalen},
  volume  = {367},
  number  = {1--2},
  year    = {2017},
  pages   = {857--879},
  doi     = {10.1007/s00208-016-1404-3}
}

@article{BenoistQuint2012,
  author  = {Benoist, Yves and Quint, Jean-Fran{\c{c}}ois},
  title   = {Random walks on finite volume homogeneous spaces},
  journal = {Inventiones Mathematicae},
  volume  = {187},
  number  = {1},
  year    = {2012},
  pages   = {37--59},
  doi     = {10.1007/s00222-011-0328-5}
}

@article{EMM1998,
  author  = {Eskin, Alex and Margulis, Gregory and Mozes, Shahar},
  title   = {Upper bounds and asymptotics in a quantitative version
             of the {Oppenheim} conjecture},
  journal = {Annals of Mathematics},
  volume  = {147},
  number  = {1},
  pages   = {93--141},
  year    = {1998},
  doi     = {10.2307/120984}
}

@incollection{EM2004,
  author    = {Eskin, Alex and Margulis, Gregory},
  title     = {Recurrence properties of random walks on finite volume
               homogeneous manifolds},
  booktitle = {Random Walks and Geometry},
  editor    = {Kaimanovich, Vadim A.},
  publisher = {Walter de Gruyter},
  address   = {Berlin and New York},
  pages     = {431--444},
  year      = {2004},
  doi       = {10.1515/9783110198089.2.431}
}

@article{EMM2015,
  author  = {Eskin, Alex and Mirzakhani, Maryam and Mohammadi, Amir},
  title   = {Isolation, equidistribution, and orbit closures for the
             {$\mathrm{SL}(2,\mathbb{R})$} action on moduli space},
  journal = {Annals of Mathematics},
  volume  = {182},
  number  = {2},
  pages   = {673--721},
  year    = {2015},
  doi     = {10.4007/annals.2015.182.2.7}
}
\end{document}